\documentclass{article}
\usepackage{amsmath,url,amsthm,afterpage}
\RequirePackage[colorlinks,citecolor=blue,urlcolor=blue]{hyperref}
\theoremstyle{plain}
\newtheorem{theorem}{Theorem}

\def\lcm{{\rm lcm}} % lcm
\allowdisplaybreaks
\title{Counting toroidal binary arrays}
\author{S. N. Ethier\thanks{Partially supported by a grant from the Simons Foundation (209632).}\\ Department of Mathematics\\ University of Utah\\ 155 South 1400 East\\ Salt Lake City, UT 84112\\ USA\\ \url{ethier@math.utah.edu}}
\date{}
\begin{document}
\maketitle

\begin{abstract}
A formula for the number of toroidal $m\times n$ binary arrays, allowing rotation of the rows and/or the columns but not reflection, is known.  Here we find a formula for the number of toroidal $m\times n$ binary arrays, allowing rotation and/or reflection of the rows and/or the columns.
\end{abstract}

\newpage

\section{Introduction}

The number of \textit{necklaces} with $n$ beads of two colors when turning over is not allowed is
\begin{equation}\label{A000031}
\frac{1}{n}\sum_{d\,|\,n}\varphi(d)\,2^{n/d},
\end{equation}
where $\varphi$ is Euler's phi function.  When turning over is allowed, the number becomes
\begin{equation}\label{A000029}
\frac{1}{2n}\sum_{d\,|\,n}\varphi(d)\,2^{n/d}+\begin{cases}2^{(n-1)/2}&\text{if $n$ is odd},\\3\cdot2^{n/2-2}&\text{if $n$ is even}.\end{cases}
\end{equation}
These are the core sequences \href{http://oeis.org/A000031}{A000031} and \href{http://oeis.org/A000029}{A000029}, respectively, in \cite{S}.

Our concern here is with two-dimensional versions of these formulas.  We consider an $m\times n$ binary array.  When opposite edges are identified, it becomes what we will call a \textit{toroidal} binary array.  Just as we can rotate a necklace without effect, we can rotate the rows and/or the columns of such an array without effect.  The number of (distinct) toroidal $m\times n$ binary arrays is
\begin{equation}\label{A184271}
\frac{1}{mn}\sum_{c\,|\,m}\;\sum_{d\,|\,n}\varphi(c)\varphi(d)\,2^{mn/\lcm(c,d)},
\end{equation}
where lcm stands for least common multiple.  This is \href{http://oeis.org/A184271}{A184271} in \cite{S}.  The diagonal is \href{http://oeis.org/A179043}{A179043}. Rows (or columns) 2--8 are \href{http://oeis.org/A184264}{A184264}--\href{http://oeis.org/A184270}{A184270}.  Row (or column) 1 is of course \href{http://oeis.org/A000031}{A000031}.

Our aim here is to find the formula that is related to (\ref{A184271}) in the same way that (\ref{A000029}) is related to (\ref{A000031}).
More precisely, we wish to count the number of toroidal $m\times n$ binary arrays allowing rotation and/or reflection of the rows and/or the columns.  At present, the resulting sequence, the diagonal, and the rows (or columns) other than the first one, are not found in \cite{S}.  Row (or column) 1 is of course \href{http://oeis.org/A000029}{A000029}.

For an alternative description, we could define a group action on the set of $m\times n$ binary arrays, which has $2^{mn}$ elements.  If the group is $C_m\times C_n$, where $C_m$ denotes the cyclic group of order $m$, then the number of orbits is given by (\ref{A184271}) (see Theorem~\ref{rotations-thm} below).  If the group is $D_m\times D_n$, where $D_m$ denotes the dihedral group of order $2m$, then the number of orbits is given by Theorem~\ref{rot-refl-thm} below.

To help clarify the distinction between the two group actions, we provide an example.  There is no distinction in the $2\times2$ case, so we consider the $3\times3$ case.  When the group is $C_3\times C_3$ (allowing rotation of the rows and/or the columns but not reflection), there are 64 orbits, as shown in Table~\ref{3x3-rotations}.

\begin{table}[htb]
\caption{\label{3x3-rotations}A list of the 64 orbits of the group action given by the group $C_3\times C_3$ acting on the set of $3\times 3$ binary arrays.  (Rows and/or columns can be rotated but not reflected.)  Each orbit is represented by its minimal element in 9-bit binary form.  Bars separate different numbers of 1s.}  
\begin{gather*}
\begin{pmatrix}0&0&0\\0&0&0\\0&0&0\end{pmatrix}\bigg|\begin{pmatrix}0&0&0\\0&0&0\\0&0&1\end{pmatrix}\bigg|
\begin{pmatrix}0&0&0\\0&0&0\\0&1&1\end{pmatrix}\begin{pmatrix}0&0&0\\0&0&1\\0&0&1\end{pmatrix}\begin{pmatrix}0&0&0\\0&0&1\\0&1&0\end{pmatrix}\begin{pmatrix}0&0&0\\0&0&1\\1&0&0\end{pmatrix}\bigg|\\
\begin{pmatrix}0&0&0\\0&0&0\\1&1&1\end{pmatrix}\begin{pmatrix}0&0&0\\0&0&1\\0&1&1\end{pmatrix}\begin{pmatrix}0&0&0\\0&0&1\\1&0&1\end{pmatrix}\begin{pmatrix}0&0&0\\0&0&1\\1&1&0\end{pmatrix}\begin{pmatrix}0&0&0\\0&1&1\\0&0&1\end{pmatrix}\begin{pmatrix}0&0&0\\0&1&1\\0&1&0\end{pmatrix}\\
\begin{pmatrix}0&0&0\\0&1&1\\1&0&0\end{pmatrix}\begin{pmatrix}0&0&1\\0&0&1\\0&0&1\end{pmatrix}\begin{pmatrix}0&0&1\\0&0&1\\0&1&0\end{pmatrix}\begin{pmatrix}0&0&1\\0&0&1\\1&0&0\end{pmatrix}\begin{pmatrix}0&0&1\\0&1&0\\1&0&0\end{pmatrix}\begin{pmatrix}0&0&1\\1&0&0\\0&1&0\end{pmatrix}\bigg|\\
\begin{pmatrix}0&0&0\\0&0&1\\1&1&1\end{pmatrix}\begin{pmatrix}0&0&0\\0&1&1\\0&1&1\end{pmatrix}\begin{pmatrix}0&0&0\\0&1&1\\1&0&1\end{pmatrix}\begin{pmatrix}0&0&0\\0&1&1\\1&1&0\end{pmatrix}\begin{pmatrix}0&0&0\\1&1&1\\0&0&1\end{pmatrix}\begin{pmatrix}0&0&1\\0&0&1\\0&1&1\end{pmatrix}\begin{pmatrix}0&0&1\\0&0&1\\1&0&1\end{pmatrix}\\
\begin{pmatrix}0&0&1\\0&0&1\\1&1&0\end{pmatrix}\begin{pmatrix}0&0&1\\0&1&0\\0&1&1\end{pmatrix}\begin{pmatrix}0&0&1\\0&1&0\\1&0&1\end{pmatrix}\begin{pmatrix}0&0&1\\0&1&0\\1&1&0\end{pmatrix}\begin{pmatrix}0&0&1\\0&1&1\\0&1&0\end{pmatrix}\begin{pmatrix}0&0&1\\1&0&0\\0&1&1\end{pmatrix}\begin{pmatrix}0&0&1\\1&0&0\\1&1&0\end{pmatrix}\bigg|\\
\begin{pmatrix}0&0&0\\0&1&1\\1&1&1\end{pmatrix}\begin{pmatrix}0&0&0\\1&1&1\\0&1&1\end{pmatrix}\begin{pmatrix}0&0&1\\0&0&1\\1&1&1\end{pmatrix}\begin{pmatrix}0&0&1\\0&1&0\\1&1&1\end{pmatrix}\begin{pmatrix}0&0&1\\0&1&1\\0&1&1\end{pmatrix}\begin{pmatrix}0&0&1\\0&1&1\\1&0&1\end{pmatrix}\begin{pmatrix}0&0&1\\0&1&1\\1&1&0\end{pmatrix}\\
\begin{pmatrix}0&0&1\\1&0&0\\1&1&1\end{pmatrix}\begin{pmatrix}0&0&1\\1&0&1\\0&1&1\end{pmatrix}\begin{pmatrix}0&0&1\\1&0&1\\1&0&1\end{pmatrix}\begin{pmatrix}0&0&1\\1&0&1\\1&1&0\end{pmatrix}\begin{pmatrix}0&0&1\\1&1&0\\0&1&1\end{pmatrix}\begin{pmatrix}0&0&1\\1&1&0\\1&0&1\end{pmatrix}\begin{pmatrix}0&0&1\\1&1&0\\1&1&0\end{pmatrix}\bigg|\\
\begin{pmatrix}0&0&0\\1&1&1\\1&1&1\end{pmatrix}\begin{pmatrix}0&0&1\\0&1&1\\1&1&1\end{pmatrix}\begin{pmatrix}0&0&1\\1&0&1\\1&1&1\end{pmatrix}\begin{pmatrix}0&0&1\\1&1&0\\1&1&1\end{pmatrix}\begin{pmatrix}0&0&1\\1&1&1\\0&1&1\end{pmatrix}\begin{pmatrix}0&0&1\\1&1&1\\1&0&1\end{pmatrix}\\
\begin{pmatrix}0&0&1\\1&1&1\\1&1&0\end{pmatrix}\begin{pmatrix}0&1&1\\0&1&1\\0&1&1\end{pmatrix}\begin{pmatrix}0&1&1\\0&1&1\\1&0&1\end{pmatrix}\begin{pmatrix}0&1&1\\0&1&1\\1&1&0\end{pmatrix}\begin{pmatrix}0&1&1\\1&0&1\\1&1&0\end{pmatrix}\begin{pmatrix}0&1&1\\1&1&0\\1&0&1\end{pmatrix}\bigg|\\
\begin{pmatrix}0&0&1\\1&1&1\\1&1&1\end{pmatrix}\begin{pmatrix}0&1&1\\0&1&1\\1&1&1\end{pmatrix}\begin{pmatrix}0&1&1\\1&0&1\\1&1&1\end{pmatrix}\begin{pmatrix}0&1&1\\1&1&0\\1&1&1\end{pmatrix}\bigg|
\begin{pmatrix}0&1&1\\1&1&1\\1&1&1\end{pmatrix}\bigg|\begin{pmatrix}1&1&1\\1&1&1\\1&1&1\end{pmatrix}
\end{gather*}
\end{table}
\afterpage{\clearpage}

When the group is $D_3\times D_3$ (allowing rotation and/or reflection of the rows and/or the columns), there are 36 orbits, as shown in  Table~\ref{3x3-rot-refl}.  

\begin{table}[hb]
\caption{\label{3x3-rot-refl}A list of the 36 orbits of the group action given by the group $D_3\times D_3$ acting on the set of $3\times 3$ binary arrays.  (Rows and/or columns can be rotated and/or reflected.)  Each orbit is represented by its minimal element in 9-bit binary form.  Bars separate different numbers of 1s.}
\begin{gather*}
\begin{pmatrix}0&0&0\\0&0&0\\0&0&0\end{pmatrix}\bigg|
\begin{pmatrix}0&0&0\\0&0&0\\0&0&1\end{pmatrix}\bigg|
\begin{pmatrix}0&0&0\\0&0&0\\0&1&1\end{pmatrix}\begin{pmatrix}0&0&0\\0&0&1\\0&0&1\end{pmatrix}\begin{pmatrix}0&0&0\\0&0&1\\0&1&0\end{pmatrix}\bigg|\\
\begin{pmatrix}0&0&0\\0&0&0\\1&1&1\end{pmatrix}\begin{pmatrix}0&0&0\\0&0&1\\0&1&1\end{pmatrix}\begin{pmatrix}0&0&0\\0&0&1\\1&1&0\end{pmatrix}\begin{pmatrix}0&0&1\\0&0&1\\0&0&1\end{pmatrix}\begin{pmatrix}0&0&1\\0&0&1\\0&1&0\end{pmatrix}\begin{pmatrix}0&0&1\\0&1&0\\1&0&0\end{pmatrix}\bigg|\\
\begin{pmatrix}0&0&0\\0&0&1\\1&1&1\end{pmatrix}\begin{pmatrix}0&0&0\\0&1&1\\0&1&1\end{pmatrix}\begin{pmatrix}0&0&0\\0&1&1\\1&0&1\end{pmatrix}\begin{pmatrix}0&0&1\\0&0&1\\0&1&1\end{pmatrix}\begin{pmatrix}0&0&1\\0&0&1\\1&1&0\end{pmatrix}\begin{pmatrix}0&0&1\\0&1&0\\0&1&1\end{pmatrix}\begin{pmatrix}0&0&1\\0&1&0\\1&0&1\end{pmatrix}\bigg|\\
\begin{pmatrix}0&0&0\\0&1&1\\1&1&1\end{pmatrix}\begin{pmatrix}0&0&1\\0&0&1\\1&1&1\end{pmatrix}\begin{pmatrix}0&0&1\\0&1&0\\1&1&1\end{pmatrix}\begin{pmatrix}0&0&1\\0&1&1\\0&1&1\end{pmatrix}\begin{pmatrix}0&0&1\\0&1&1\\1&0&1\end{pmatrix}\begin{pmatrix}0&0&1\\0&1&1\\1&1&0\end{pmatrix}\begin{pmatrix}0&0&1\\1&1&0\\1&1&0\end{pmatrix}\bigg|\\
\begin{pmatrix}0&0&0\\1&1&1\\1&1&1\end{pmatrix}\begin{pmatrix}0&0&1\\0&1&1\\1&1&1\end{pmatrix}\begin{pmatrix}0&0&1\\1&1&0\\1&1&1\end{pmatrix}\begin{pmatrix}0&1&1\\0&1&1\\0&1&1\end{pmatrix}\begin{pmatrix}0&1&1\\0&1&1\\1&0&1\end{pmatrix}\begin{pmatrix}0&1&1\\1&0&1\\1&1&0\end{pmatrix}\bigg|\\
\begin{pmatrix}0&0&1\\1&1&1\\1&1&1\end{pmatrix}\begin{pmatrix}0&1&1\\0&1&1\\1&1&1\end{pmatrix}\begin{pmatrix}0&1&1\\1&0&1\\1&1&1\end{pmatrix}\bigg|
\begin{pmatrix}0&1&1\\1&1&1\\1&1&1\end{pmatrix}\bigg|
\begin{pmatrix}1&1&1\\1&1&1\\1&1&1\end{pmatrix}
\end{gather*}
\end{table}

Our interest in the number of toroidal $m\times n$ binary arrays allowing rotation and/or reflection of the rows and/or the columns derives from the fact that this is the size of the state space of the projection of the Markov chain in \cite{MR} under the mapping that takes a state to the orbit containing it.  This reduction of the state space, from 512 states to 36 states in the $3\times3$ case for example, makes computation easier.

\section{Rotations of rows and columns}

Let $X_{m,n}:=\{0,1\}^{\{0,1,\ldots,m-1\}\times\{0,1,\ldots,n-1\}}$ be the set of $m\times n$ arrays of 0s and 1s, of which there are $2^{mn}$.  Let $a(m,n)$ denote the number of orbits of the group action on $X_{m,n}$ by the group $C_m\times C_n$.  In other words, $a(m,n)$ is the number of (distinct) toroidal $m\times n$ binary arrays, allowing rotation of the rows and/or the columns but not reflection.

\begin{theorem}\label{rotations-thm}
\begin{equation}\label{a(m,n)}
a(m,n)=\frac{1}{mn}\sum_{c\,|\,m}\;\sum_{d\,|\,n}\varphi(c)\varphi(d)\,2^{mn/\lcm(c,d)}.
\end{equation}
\end{theorem}

\begin{proof}
By the P\'olya enumeration theorem \cite{Wc},
\begin{equation}\label{a(m,n)-PET}
a(m,n)=\frac{1}{mn}\sum_{i=0}^{m-1}\sum_{j=0}^{n-1}2^{A_{ij}},
\end{equation}
where $A_{ij}$ is the number of cycles in the permutation $\sigma^i\tau^j$; here $\sigma$ rotates the rows (row 0 becomes row 1, row 1 becomes row 2, \dots, row $m-1$ becomes row 0) and $\tau$ rotates the columns.  For example, $A_{00}=mn$ because the identity permutation has $mn$ fixed points, each of which is a cycle of length 1.

It is well known that, if $d$ divides $n$, then the number of elements of $C_n$ that are of order $d$ is $\varphi(d)$.  So if $c$ divides $m$ and $d$ divides $n$, then the number of pairs $(i,j)$ such that $\sigma^i$ is of order $c$ and $\tau^j$ is of order $d$ is $\varphi(c)\varphi(d)$.  And if $\sigma^i$ is of order $c$ and $\tau^j$ is of order $d$, then $\sigma^i\tau^j$ is of order $\lcm(c,d)$ because $\sigma^i$ and $\tau^j$ commute.  Consequently, $\lcm(c,d)$ is the length of each cycle of the permutation $\sigma^i\tau^j$, so $A_{ij}=mn/\lcm(c,d)$, and (\ref{a(m,n)}) follows from (\ref{a(m,n)-PET}).
\end{proof}

Clearly, $a(1,n)$ reduces to (\ref{A000031}).  Table~\ref{array-a} provides numerical values of $a(m,n)$ for small $m$ and $n$.

\begin{table}[htb]
\caption{\label{array-a}The number $a(m,n)$ of toroidal $m\times n$ binary arrays, allowing rotation of the rows and/or the columns but not reflection, for $m,n=1,2,\ldots,8$.}
\catcode`@=\active \def@{\hphantom{0}}
\tabcolsep=1.5mm
\begin{center}
\begin{tiny}
\begin{tabular}{rrrrrrrr}
\hline
\noalign{\smallskip}
2 & 3 & 4 & 6 & 8 & 14 & 20 & 36\\ 3 & 7 & 14 & 40 & 108 & 362 & 1182 & 
  4150\\ 4 & 14 & 64 & 352 & 2192 & 14624 & 99880 & 699252\\ 6 & 40 & 352 & 
  4156 & 52488 & 699600 & 9587580 & 134223976\\ 8 & 108 & 2192 & 52488 & 
  1342208 & 35792568 & 981706832 & 27487816992\\ 14 & 362 & 14624 & 699600 &
   35792568 & 1908897152 & 104715443852 & 5864063066500\\ 20 & 1182 & 
  99880 & 9587580 & 981706832 & 104715443852 & 11488774559744 & 
  1286742755471400\\ 36 & 4150 & 699252 & 134223976 & 27487816992 & 
  5864063066500 & 1286742755471400 & 288230376353050816\\
\noalign{\smallskip}
\hline
\end{tabular}
\end{tiny}
\end{center}
\end{table}

\section{Rotations and reflections of rows and columns}

Let $b(m,n)$ denote the number of orbits of the group action on $X_{m,n}$ by the group $D_m\times D_n$.  In other words, $b(m,n)$ is the number of (distinct) toroidal $m\times n$ binary arrays, allowing rotation and/or reflection of the rows and/or the columns.

\begin{theorem}\label{rot-refl-thm}
\begin{equation}
b(m,n)=b_1(m,n)+b_2(m,n)+b_3(m,n)+b_4(m,n),
\end{equation}
where
$$
b_1(m,n)=\frac{1}{4mn}\sum_{c\,|\,m}\;\sum_{d\,|\,n}\varphi(c)\varphi(d)\,2^{mn/\lcm(c,d)},
$$ 
\begin{eqnarray*}
&&\!\!\!\!\!\!\!\!b_2(m,n)\\
&=&\begin{cases}(4n)^{-1}2^{(m+1)n/2}&\text{if $m$ is odd}\\(8n)^{-1}[2^{mn/2}+2^{(m+2)n/2}]&\text{if $m$ is even}\end{cases}\;\;+\frac{1}{4n}\sum_{d\ge2:\,d\,|\,n}\varphi(d)\,2^{mn/d}\\
&&{}+\begin{cases}(4n)^{-1}\sum'[2^{(m + 1)\gcd(j, n)/2} - 2^{m \gcd(j, n)}]&\text{if $m$ is odd}\\(8n)^{-1}\sum'[2^{m\gcd(j, n)/2}+2^{(m+2)\gcd(j, n)/2}- 2^{m \gcd(j, n)+1}]&\text{if $m$ is even}\end{cases}
\end{eqnarray*}
with $\sum'=\sum_{1\le j\le n-1:\,n/\gcd(j,n){\rm\ is\ odd}}$,
$$b_3(m,n)=b_2(n,m),$$ 
and
$$
b_4(m,n)=\begin{cases}2^{(mn - 3)/2}&\text{if $m$ and $n$ are odd},\\
3\cdot2^{mn/2 - 3}&\text{if $m$ and $n$ have opposite parity},\\
7\cdot2^{m n/2 - 4}&\text{if $m$ and $n$ are even}.\end{cases}
$$
\end{theorem}

\begin{proof}
Again by the P\'olya enumeration theorem \cite{Wc},
\begin{eqnarray}\label{b(m,n)-PET}
b(m,n)&=&\frac{1}{4mn}\sum_{i=0}^{m-1}\sum_{j=0}^{n-1}[2^{A_{ij}}+2^{B_{ij}}+2^{C_{ij}}+2^{D_{ij}}],
\end{eqnarray}
where $A_{ij}$ (resp., $B_{ij}$, $C_{ij}$, $D_{ij}$) is the number of cycles in the permutation $\sigma^i\tau^j$ (resp., $\sigma^i\tau^j\rho$, $\sigma^i\tau^j\theta$, $\sigma^i\tau^j\rho\theta$); here $\sigma$ rotates the rows (row 0 becomes row 1, row 1 becomes row 2, \dots, row $m-1$ becomes row 0), $\tau$ rotates the columns, $\rho$ reflects the rows (rows 0 and $m-1$ are interchanged, rows 1 and $m-2$ are interchanged, \dots, rows $\lfloor m/2\rfloor-1$ and $m-\lfloor m/2\rfloor$ are interchanged), and $\theta$ reflects the columns.

By the proof of Theorem~\ref{rotations-thm}, we know the form of $A_{ij}$, and this gives the formula for $b_1(m,n)$.

Next we find $(B_{i0})$, the entries in the 0th column of matrix $B$.  For $i=0,1,\ldots,m-1$, the permutation $\sigma^i\rho$ can be described by its effect on the rows of $\{0,1,\ldots,m-1\}\times\{0,1,\ldots,n-1\}$.  It reverses the first $m-i$ rows and reverses the last $i$ rows.  Since the reversal of $k$ consecutive integers has $k/2$ transpositions if $k$ is even and $(k-1)/2$ transpositions and one fixed point if $k$ is odd, the permutation of $\{0,1,\ldots,m-1\}$ induced by $\sigma^i\rho$ has $(m-1)/2$ transpositions and one fixed point if $m$ is odd, and $m/2$ transpositions if $i$ is even and $m$ is even, and $(m-2)/2$ transpositions and two fixed points if $i$ is odd and $m$ is even.  These numbers must be multiplied by $n$ for the permutation $\sigma^i\rho$ of $\{0,1,\ldots,m-1\}\times\{0,1,\ldots,n-1\}$.  The results are that $B_{i0}=(m+1)n/2$ if $m$ is odd, $B_{i0}=mn/2$ if $i$ is even and $m$ is even, and $B_{i0}=(m+2)n/2$ if $i$ is odd and $m$ is even.  Therefore, $(4mn)^{-1}\sum_{i=0}^{m-1}2^{B_{i0}}=(4n)^{-1}2^{(m+1)n/2}$ if $m$ is odd, whereas $(4mn)^{-1}\sum_{i=0}^{m-1}2^{B_{i0}}=(8n)^{-1}[2^{mn/2}+2^{(m+2)n/2}]$ if $m$ is even, and this gives the first term in the formula for $b_2(m,n)$.

We turn to $B_{ij}$ for $i=0,1,\ldots,m-1$ and $j=1,2,\ldots,n-1$.  First, by a property of cyclic groups, $\tau^j$ has order $d:=n/\gcd(j,n)$.  If $d$ is even, then, since $\sigma^i\rho$ has order 2 (see the preceding paragraph), $\sigma^i\tau^j\rho$ has order $d$ and all of its cycles have length $d$.  In this case, $B_{ij}=mn/d=m\gcd(j,n)$.  Suppose then that $d$ is odd.  There are three cases: ($i$) $m$ odd, ($ii$) $i$ even and $m$ even, and ($iii$) $i$ odd and $m$ even.  Recall that $\sigma^i\rho$ reverses the first $m-i$ rows and reverses the last $i$ rows.  In case ($i$), one row is fixed by $\sigma^i\rho$, so cycles of $\sigma^i\tau^j\rho$ in this row have length $d$ and all others have length $2d$.  We find that $B_{ij}=n/d+(m-1)n/(2d)=(m+1)n/(2d)=(m+1)\gcd(j,n)/2$.  In case ($ii$), no rows are fixed by $\sigma^i\rho$, so all cycles of $\sigma^i\tau^j\rho$ have length $2d$.  It follows that $B_{ij}=mn/(2d)=m\gcd(j,n)/2$.  In case ($iii$), two rows are fixed by $\sigma^i\rho$, so cycles of $\sigma^i\tau^j\rho$ in these rows have length $d$ and all others have length $2d$.  We conclude that $B_{ij}=2n/d+(m-2)n/(2d)=(m+2)\gcd(j,n)/2$.  If the formula for $B_{ij}$ that holds when $d$ is even were valid generally, we would have the second term in the formula for $b_2(m,n)$.  The third term in the formula for $b_2(m,n)$ is a correction to the second term to treat the cases ($i$)--($iii$) in which $d$ is odd.  

Next, the formula for $b_3(m,n)$ follows by symmetry.  More explicitly,
$$
b_3(m,n)=\frac{1}{4mn}\sum_{i=0}^{m-1}\sum_{j=0}^{n-1}2^{C_{ij}}=\frac{1}{4nm}\sum_{j=0}^{n-1}\sum_{i=0}^{m-1}2^{B_{ji}}=b_2(n,m)
$$
because the number of cycles $C_{ij}$ of $\sigma^i\tau^j\theta$ for an $m\times n$ array is equal to the number of cycles $B_{ji}$ of $\sigma^j\tau^i\rho$ for an $n\times m$ array.

Finally, we consider $b_4(m,n)$.  For $i=0,1,\ldots,m-1$ and $j=0,1,\ldots,n-1$, $\sigma^i\tau^j\rho\theta$ has the effect of reversing the first $m-i$ rows, reversing the last $i$ rows, reversing the first $n-j$ columns, and reversing the last $j$ columns.  If $m$ and $n$ are odd, then there is one fixed point and $(mn-1)/2$ transpositions, so $D_{ij}=(mn+1)/2$ for all $i$ and $j$, hence $b_4(m,n)=2^{(mn-3)/2}$.  If $m$ is odd and $n$ is even, then $D_{ij}=mn/2$ for all $i$ and even $j$ and $D_{ij}=mn/2+1$ for all $i$ and odd $j$.  This leads to $b_4(m,n)=(1/8)[2^{mn/2}+2^{mn/2+1}]=3\cdot 2^{mn/2-3}$.  If $m$ is even and $n$ is odd, then $D_{ij}=mn/2$ for even $i$ and all $j$ and $D_{ij}=mn/2+1$ for odd $i$ and all $j$.  This leads to the same formula for $b_4(m,n)$.  Finally, if $m$ and $n$ are even, then $D_{ij}=mn/2$ unless $i$ and $j$ are both odd, in which case $D_{ij}=mn/2+2$.  This implies that $b_4(m,n)=(1/4)[(3/4)2^{mn/2}+(1/4)2^{mn/2+2}]=7\cdot2^{mn/2-4}$.

This completes the proof.
\end{proof}

It is easy to check that $b(1,n)$ reduces to (\ref{A000029}).  Table~\ref{array-b} provides numerical values of $b(m,n)$ for small $m$ and $n$.

\begin{table}[ht]
\caption{\label{array-b}The number $b(m,n)$ of toroidal $m\times n$ binary arrays, allowing rotation and/or reflection of the rows and/or the columns, for $m,n=1,2,\ldots,8$.}
\catcode`@=\active \def@{\hphantom{0}}
\tabcolsep=1.5mm
\begin{center}
\begin{tiny}
\begin{tabular}{rrrrrrrr}
\hline
\noalign{\smallskip}
2 & 3 & 4 & 6 & 8 & 13 & 18 & 30\\ 3 & 7 & 13 & 34 & 78 & 237 & 687 & 2299\\ 4 & 
  13 & 36 & 158 & 708 & 4236 & 26412 & 180070\\ 6 & 34 & 158 & 1459 & 14676 & 
  184854 & 2445918 & 33888844\\ 8 & 78 & 708 & 14676 & 340880 & 8999762 & 
  245619576 & 6873769668\\ 13 & 237 & 4236 & 184854 & 8999762 & 478070832 & 
  26185264801 & 1466114420489\\ 18 & 687 & 26412 & 2445918 & 245619576 & 
  26185264801 & 2872221202512 & 321686550498774\\ 30 & 2299 & 180070 & 
  33888844 & 6873769668 & 1466114420489 & 321686550498774 & 
  72057630729710704\\
\noalign{\smallskip}
\hline
\end{tabular}
\end{tiny}
\end{center}
\end{table}

\end{document}